\documentclass[a4paper,12pt]{amsart}
\usepackage{amssymb}
\usepackage[T1]{fontenc}
\usepackage{amsfonts}
\usepackage[utf8]{inputenc}
\usepackage{amsmath}
\usepackage{amsthm}
\usepackage{hyperref}

\numberwithin{equation}{section}

\newtheorem{theorem}{Theorem}[section]
\newtheorem{lemma}[theorem]{Lemma}
\newtheorem{corollary}[theorem]{Corollary}

\theoremstyle{definition}
\newtheorem{definition}[theorem]{Definition}

\newtheorem{remark}[theorem]{Remark}

\renewcommand{\epsilon}{\varepsilon}

\renewcommand{\c}{\mathcal{C}}

\newcommand{\A}{\mathcal{A}}
\newcommand{\C}{\mathbb{C}}
\newcommand{\F}{\mathcal{F}}
\newcommand{\M}{\operatorname{M}}
\newcommand{\N}{\mathbb{N}}

\newcommand{\R}{\mathbb{R}}
\newcommand{\X}{\mathcal{X}}

\newcommand{\id}{\operatorname{id}}
\newcommand{\1}{{1}}
\newcommand{\GL}{\operatorname{GL}}
\renewcommand{\Im}{\operatorname{Im}}

\begin{document}

\title[Free Berry-Esseen theorems]{Operator-valued and multivariate free Berry-Esseen theorems}

\author[T. Mai and R. Speicher]{Tobias Mai$^*$ and Roland Speicher$^\dagger$}

\thanks{$^*\,$Research supported by funds of Roland Speicher from the Alfried Krupp von Bohlen und Halbach-Stiftung}

\thanks{$^\dagger\,$Research partially supported by a Discovery grant from NSERC (Canada) and by
a Killam Fellowship from the Canada Council for the Arts}

\address{Saarland University, Fachbereich Mathematik, Postfach 151150,
66041 Saarbr\"ucken, Germany}

\dedicatory{ Dedicated to Professor Friedrich G\"otze\\ on the occassion of his
60th birthday}

\maketitle

\begin{abstract}
We address the question of a Berry-Esseen type theorem for the speed of convergence in a multivariate free central limit theorem. For this, we estimate the difference between the operator-valued Cauchy transforms of the normalized partial sums in an operator-valued free central limit theorem and the Cauchy transform of the limiting operator-valued semicircular element. Since we have to deal with in general non-self-adjoint operators, we introduce the notion of matrix-valued resolvent sets and study the behavior of Cauchy transforms on them.
\end{abstract}



\section{Introduction}

In classical probability theory the famous Berry-Esseen theorem gives a quantitative statement about the order of convergence in the central limit theorem. It states in its simplest version: If $(X_i)_{i\in\N}$ is a sequence of independent and identically distributed random variables with mean $0$ and variance $1$, then the distance between $S_n := \frac{1}{\sqrt{n}}(X_1 + \dots + X_n)$ and a normal variable $\gamma$ of mean $0$ and variance $1$ can be estimated in terms of the Kolmogorov distance $\Delta$ by
$$\Delta(S_n,\gamma) \leq C \frac{1}{\sqrt{n}} \rho,$$
where $C$ is a constant and $\rho$ is the absolute third moment of the variables $X_i$. The question for a free analogue of the Berry-Esseen estimate in the case of one random variable was answered by Christyakov and G\"otze in \cite{Goetze} (and independently, under the more restrictive assumption of compact support of the $X_i$, by Kargin \cite{Kargin}): If $(X_i)_{i\in\N}$ is a sequence of free and identically distributed variables with mean $0$ and variance $1$, then the distance between $S_n := \frac{1}{\sqrt{n}}(X_1 + \dots + X_n)$ and a semicircular variable $s$ of mean $0$ and variance $1$ can be estimated as
$$\Delta(S_n,s) \leq c \frac{|m_3|+\sqrt{m_4}}{\sqrt{n}},$$ 
where $c>0$ is an absolute constant and $m_3$ and $m_4$ are the third and fourth moment, respectively, of the $X_i$.

In this paper we want to present an approach to a multivariate version of a free Berry-Esseen theorem. The general idea is the following: Since there is up to now no suitable replacement of the Kolmorgorov metric in the multivariate case, we will, in order to describe the speed of convergence of a $d$-tuple $(S_n^{(1)},\dots,S_n^{(d)})$ of partial sums to the limiting semicircular family $(s_1,\dots,s_d)$, consider the speed of convergence of $p(S_n^{(1)},\dots,S_n^{(d)})$ to $p(s_1,\dots,s_d)$ for any self-adjoint polynomial $p$ in $d$ non-commuting variables. By using the linearization trick of Haagerup and Thorbj\o rnsen \cite{Haagerup1, Haagerup2}, we can reformulate this in an operator-valued setting, where we will state an operator-valued free Berry-Esseen theorem. Because estimates for the difference between scalar-valued Cauchy transforms translate by results of Bai \cite{Bai} to estimates with respect to the Kolmogorov distance, it is convenient to describe the speed of convergence in terms of Cauchy transforms.
On the level of deriving equations for the (operator-valued) Cauchy
transforms we can follow ideas which are used for dealing with speed of convergence
questions for random matrices; here we are inspired in particular by the work of G\"otze
and Tikhomirov \cite{GT}, but see also \cite{Bai}.

Since the transition from the multivariate to the operator-valued setting leads to operators which are, even if we start from selfadjoint polynomials $p$, in general not self-adjoint, we have to deal with (operator-valued) Cauchy transforms defined on domains different from the usual ones. Since most of the analytic tools fail in this generality, we have to develop them along the way.

As a first step in this direction, the present paper (which is based on the unpublished preprint \cite{Speicher5}) leads finally to the proof of the following theorem:

\begin{theorem}\label{Berry-Esseen}
Let $(\c,\tau)$ be a non-commutative $C^\ast$-probability space with $\tau$ faithful and put $\A:=\M_m(\C)\otimes\c$ and $E:=\id\otimes\tau$. Let $(X_i)_{i\in\N}$ be a sequence of non-zero elements in the operator-valued probability space $(\A,E)$. We assume:
\begin{itemize}
 \item All $X_i$'s have the same $\ast$-distribution with respect to $E$ and their first moments vanish, i.e. $E[X_i]=0$.
 \item The $X_i$ are $*$-free with amalgamation over $\M_m(\C)$ (which means that the $\ast$-algebras $\X_i$, generated by $\M_m(\C)$ and $X_i$, are free with respect to $E$).
 \item We have $\displaystyle{\sup_{i\in\N} \|X_i\| < \infty}$.
\end{itemize}
Then the sequence $(S_n)_{n\in\N}$ defined by
$$S_n := \frac{1}{\sqrt{n}} \sum^n_{i=1} X_i,\qquad n\in\N$$
converges to an operator-valued semicircular element $s$ . Moreover, we can find $\kappa > 0$, $c>1$, $C>0$ and $N\in\N$ such that
$$\|{G}_{s}(b) - {G}_{S_n}(b)\| \leq C \frac{1}{\sqrt{n}}\|b\| \qquad\text{for all $b\in \Omega$ and $n\geq N$},$$
where
$$\Omega := \Big\{b\in\GL_m(\C)\mid \|b^{-1}\| < \kappa,\ \|b\|\cdot\|b^{-1}\| < c\Big\}$$
and where $G_s$  and $G_{S_n}$ denote the operator-valued Cauchy transforms of $s$ and of $S_n$, respectively.
\end{theorem}

Applying this operator-valued statement to our multivariate problem gives the following main result on a multivariate free Berry Esseen theorem.

\begin{theorem}\label{Berry-Esseen-multivariate}
Let $(x_i^{(k)})_{k=1}^d$, $i\in\N$, be free and identically distributed sets of $d$ self-adjoint non-zero random variables in some non-commutative $C^\ast$-probability space $(\c,\tau)$, with $\tau$ faithful, such that the conditions
$$\tau(x_i^{(k)}) = 0 \qquad\text{for $k=1,\dots,d$ and all $i\in\N$}$$
and
$$\sup_{i\in\N} \max_{k=1,\dots,d} \|x_i^{(k)}\| < \infty$$
are fulfilled. We denote by $\Sigma =(\sigma_{k,l})_{k,l=1}^d$, where $\sigma_{k,l} := \tau(x_i^{(k)} x_i^{(l)})$, their joint covariance matrix. Moreover, we put
$$S_n^{(k)} := \frac{1}{\sqrt{n}} \sum^n_{i=1} x_i^{(k)} \qquad\text{for $k=1,\dots,d$ and all $n\in\N$}.$$
Then  $(S_n^{(1)},\dots,S_n^{(d)})$ converges in distribution to a semicircular family $(s_1,\dots,s_d)$ of covariance $\Sigma$. We can quantify the speed of convergence in the following way. Let $p$ be  a (not necessarily self-adjoint) polynomial in $d$ non-commutating variables and put
$$P_n := p(S^{(1)}_n,\dots,S^{(d)}_n) \qquad\text{and}\qquad P := p(s_1,\dots,s_d).$$
Then, there are constants $C>0$, $R>0$ and $N\in\N$ (depending on the polynomial) such that
$$|G_P(z) - G_{P_n}(z)| \leq C \frac{1}{\sqrt{n}} \qquad\text{for all $|z|>R$ and $n\geq N$},$$
where $G_P$ and $G_{P_n}$ denote the scalar-valued Cauchy transform of 
$P$ and of $P_n$, respectively.
\end{theorem}

In order to deduce estimates for the Kolmogorov distance $\Delta(P_n,P)$  one has to transfer the estimate for the difference of the scalar-valued Cauchy transforms of $P_n$ and $P$ from near infinity to a neigborhood of the real axis. A partial solution to this problem was given in the appendix of \cite{Speicher4}, which we will recall in Section 4. But this leads to the still unsolved question, whether $p(s_1,\dots,s_d)$ has a continuous density. We conjecture that the latter is true for any selfadjoint polynomial in free semicirculars, but at present we are not aware of a proof of that statement.

The paper is organized as follows. In Section 2 we recall some basic facts about holomorphic functions on domains in Banach spaces. The tools to deal with matrix-valued Cauchy transform will be presented in Section 3. Section 4 is devoted to the proof of Theorem \ref{Berry-Esseen}  and Theorem \ref{Berry-Esseen-multivariate}.

\section{Holomorphic functions on domains in Banach spaces}

For reader's convenience, we briefly recall the definition of holomorphic functions on domains in Banach spaces and we state the theorem of Earle-Hamilton, which will play a major role in the subsequent sections.

\begin{definition}
Let $(X,\|\cdot\|_X)$, $(Y,\|\cdot\|_Y)$ be two complex Banach spaces and let $D\subseteq X$ be an open subset of $X$. A function $f: D\rightarrow Y$ is called
\begin{itemize}
 \item \textbf{strongly holomorphic}, if for each $x\in D$ there exists a bounded linear mapping $Df(x): X\rightarrow Y$ such that $$\lim_{y\rightarrow 0} \frac{\|f(x+y)-f(x)-Df(x)y\|_Y}{\|y\|_X} = 0.$$
 \item \textbf{weakly holomorphic}, if it is locally bounded and the mapping $$\lambda\mapsto \phi(f(x+\lambda y))$$ is holomorphic at $\lambda=0$ for each $x\in D$, $y\in Y$ and all continuous linear functionals $\phi: Y\rightarrow\C$.
\end{itemize}
\end{definition}

An important theorem due to Dunford says, that a function on a domain (i.e. an open and connected subset) in a Banach space is strongly holomorphic if and only if it is weakly holomorphic. Hence, we do not have to distinguish between both definitions.

\begin{definition}
Let $D$ be a nonempty domain in a complex Banach space $(X,\|\cdot\|)$ and let $f: D\rightarrow D$ be a holomorphic function. We say, that \textbf{$f(D)$ lies strictly inside $D$}, if there is some $\epsilon>0$ such that
$$B_\epsilon(f(x)) \subseteq D \qquad\text{for all $x\in D$}$$
holds, whereby we denote by $B_r(y)$ the open ball with radius $r$ around $y$.
\end{definition}

The remarkable fact, that strict holomorphic mappings are strict contractions in the so-called Carath\'eodory-Riffen-Finsler metric, leads to the following theorem of Earle-Hamilton (cf. \cite{Earle}), which can be seen as a holomorphic version of Banach's contraction mapping theorem. For a proof of this theorem and variations of the statement we refer to \cite{Harris}.

\begin{theorem}[Earle-Hamilton, 1970]\label{Earle-Hamilton}
Let $\emptyset\neq D\subseteq X$ be a domain in a Banach space $(X,\|\cdot\|)$ and let $f: D\rightarrow D$ be a bounded holomorphic function. If $f(D)$ lies strictly inside $D$, then $f$ has a unique fixed point in $D$.
\end{theorem}

\section{Matrix-valued spectra and Cauchy transforms}

The statement of the following lemma is well-known and quite simple. But since it turns out to be extremely helpful, it is convenient to recall it here.

\begin{lemma}\label{norm-inverse}
Let $(A,\|\cdot\|)$ be a complex Banach-algebra with unit $\1$. If $x\in A$ is invertible and $y\in A$ satisfies $\|x-y\| < \sigma\frac{1}{\|x^{-1}\|}$ for some $0<\sigma<1$, then $y$ is invertible as well and we have
$$\|y^{-1}\| \leq \frac{1}{1-\sigma} \|x^{-1}\|.$$ 
\end{lemma}

\begin{proof}
We can easily check that
$$\sum^\infty_{n=0} \big(x^{-1}(x-y)\big)^nx^{-1}$$
is absolutely convergent in $A$ and gives the inverse element of $y$. Moreover we get
$$\|y^{-1}\| \leq \sum^\infty_{n=0} \big(\|x^{-1}\|\|x-y\|\big)^n \|x^{-1}\| < \frac{1}{1-\sigma}\|x^{-1}\|,$$
which proves the stated estimate.
\end{proof}

Let $(\c,\tau)$ be a non-commutative $C^\ast$-probability space, i.e., $\c$ is a unital $C^*$-algebra and $\tau$ is a unital state (positive linear functional) on $\c$; we will always assume that $\tau$ is faithful. For fixed $m\in\N$ we define the operator-valued $C^\ast$-probability space $\A:=\M_m(\C)\otimes\c$ with conditional expectation
$$E:=\id_m\otimes\tau:\ \A\rightarrow\M_m(\C),\ b\otimes c \mapsto \tau(c)b,$$
where we denote by $\M_m(\C)$ the $C^\ast$-algebra of all $m\times m$ matrices over the complex numbers $\C$. 
Under the canonical identification of $\M_m(\C)\otimes\c$ with $\M_m(\c)$
(matrices with entries in $\c$), the expectation $E$ corresponds to applying the state $\tau$ entrywise in a matrix. We will also identify $b\in \M_m(\C)$ with $b\otimes 1\in \A$.

\begin{definition}
For $a\in\A=\M_m(\c)$ we define the \textbf{matrix-valued resolvent set}
$$\rho_m(a) := \{b\in \M_m(\C)\mid \text{$b- a$ is invertible in $\A$}\}$$
and the \textbf{matrix-valued spectrum}
$$\sigma_m(a) := \M_m(\C) \backslash \rho_m(a).$$
\end{definition}

Since the set $\GL(\A)$ of all invertible elements in $\A$ is an open subset of $\A$ (cf. Lemma \ref{norm-inverse}), the continuity of the mapping
$$f_a:\ \M_m(\C)\rightarrow\A,\ b\mapsto b - a$$
implies, that the matrix-valued resolvent set $\rho_m(a) = f^{-1}_a(\GL(\A))$ of an element $a\in\A$ is an open subset of $\M_m(\C)$. Hence, the matrix-valued spectrum $\sigma_m(a)$ is always closed.

Although the behavior of this matrix-valued generalizations of the classical resolvent set and spectrum seems to be quite similar to the classical case (which is of course included in our definition for $m=1$), the matrix valued spectrum is in general not bounded and hence not a compact subset of $\M_m(\C)$. For example, we have for all $\lambda\in\C$, that
$$\sigma_m(\lambda\1) = \{b\in\M_m(\C)\mid \lambda\in\sigma_{\M_m(\C)}(b)\},$$
i.e. $\sigma_m(\lambda\1)$ consists of all matrices $b\in\M_m(\C)$ for which $\lambda$ belongs to the spectrum $\sigma_{\M_m(\C)}(b)$. Particularly, $\sigma_m(\lambda\1)$ is unbounded for $m\geq 2$.

In the following, we denote by $\GL_m(\C) := \GL(\M_m(\C))$ the set of all invertible matrices in $\M_m(\C)$.

\begin{lemma}\label{spectra-inclusion}
Let $a\in\A$ be given. Then for all $b\in\GL_m(\C)$ the following inclusion holds:
$$\big\{\lambda b\mid \lambda\in\rho_\A(b^{-1}a)\big\} \subseteq \rho_m(a)$$
\end{lemma}

\begin{proof}
Let $\lambda\in\rho_\A(b^{-1}a)$ be given. By definition of the usual resolvent set this means that $\lambda\1 - b^{-1}a$ is invertible in $\A$. It follows, that
$$\lambda b- a = b \big(\lambda 1 - b^{-1}a\big)$$
is invertible as well, and we get, as desired, $\lambda b\in \rho_m(a)$.
\end{proof}

\begin{lemma}\label{resolvent}
For all $0\neq a\in\A$ we have
$$\Big\{b\in\GL_m(\C)\mid \|b^{-1}\| < \frac{1}{\|a\|}\Big\} \subseteq \rho_m(a)$$
and
$$\sigma_m(a) \cap \GL_m(\C) \subseteq \Big\{b\in\GL_m(\C)\mid \|b^{-1}\| \geq \frac{1}{\|a\|}\Big\}.$$
\end{lemma}

\begin{proof}
Obviously, the second inclusion is a direct consequence of the first. Hence, it suffices to show the first statement.

Let $b\in\GL_m(\C)$ with $\|b^{-1}\| < \frac{1}{\|a\|}$ be given. It follows, that $h := \1 - b^{-1}a$ is invertible, because
$$\|\1-h\| = \|b^{-1}a\| \leq \|b^{-1}\| \cdot\|a\| < 1.$$
Therefore, we can deduce, that also
\begin{equation}\label{norm-start}
b - a = b\big(1 - b^{-1}a\big)
\end{equation}
is invertible, i.e. $b\in\rho_m(a)$. This proves the assertion.
\end{proof}

The main reason to consider matrix-valued resolvent sets is, that they are the natural domains for matrix-valued Cauchy transforms, which we will define now.

\begin{definition}
For $a\in\A$ we call
$${G}_a:\ \rho_m(a) \rightarrow \M_m(\C),\ b\mapsto E\big[(b - a)^{-1}\big]$$
the \textbf{matrix-valued Cauchy transform} of $a$.
\end{definition}

Note that ${G}_a$ is a continuous function (and hence locally bounded) and induces for all $b_0\in \rho_m(a)$, $b\in\M_m(\C)$ and bounded linear functionals $\phi: \A\rightarrow\C$ a function 
$$\lambda \mapsto \phi\big({G}_a(b_0+\lambda b)\big),$$
which is holomorphic in a neighborhood of $\lambda=0$. Hence, ${G}_a$ is weakly holomorphic and therefore (as we have seen in the previous section) strongly holomorphic as well.

Because the structure of $\rho_m(a)$ and therefore the behavior of ${G}_a$ might in general be quite complicated, we restrict our attention to a suitable restriction of ${G}_a$. In this way, we will get some additional properties of ${G}_a$.

The first restriction enables us to control the norm of the matrix-valued Cauchy transform on a sufficiently nice subset of the matrix-valued resolvent set. 

\begin{lemma}\label{bounded}
Let $0\neq a\in\A$ be given. For $0<\theta<1$ the matrix valued Cauchy transform ${G}_a$ induces a mapping
$${G}_a:\ \Big\{b\in\GL_m(\C)\mid \|b^{-1}\| < \theta\cdot\frac{1}{\|a\|}\Big\} \rightarrow \Big\{b\in\M_m(\C)\mid \|b\| < \frac{\theta}{1-\theta}\cdot \frac{1}{\|a\|}\Big\}.$$
\end{lemma}

\begin{proof}
Lemma \ref{resolvent} (c) tells us, that the open set
$$U:=\Big\{b\in\GL_m(\C)\mid \|b^{-1}\| < \theta\cdot\frac{1}{\|a\|}\Big\}$$
is contained in $\rho_m(a)$, i.e. ${G}_a$ is well-defined on $U$. Moreover, we get from \eqref{norm-start}
$$(b-a)^{-1} = \big(\1 - b^{-1} a\big)^{-1}b^{-1} = \sum^\infty_{n=0} \big(b^{-1}a\big)^nb^{-1}$$
and hence
\begin{equation}\label{estimates-0}
\|{G}_a(b)\| \leq \|(b-a)^{-1}\| \leq \|b^{-1}\| \sum^\infty_{n=0} \big(\|b^{-1}\|\|a\|\big)^n < \frac{\theta}{1-\theta}\cdot \frac{1}{\|a\|}
\end{equation}
for all $b\in U$. This proves the claim.
\end{proof}

To ensure, that the range of ${G}_a$ is contained in $\GL_m(\C)$, we have to shrink the domain again. 

\begin{lemma}\label{invertible}
Let $0\neq a\in\A$ be given. For $0<\theta<1$ and $c>1$ we define
$$\Omega := \Big\{b\in\GL_m(\C)\mid \|b^{-1}\| < \theta\cdot\frac{1}{\|a\|},\ \|b\|\cdot\|b^{-1}\| < c\Big\}$$
and
$$\Omega' := \Big\{b\in\GL_m(\C)\mid \|b\| < \frac{\theta}{1-\theta}\cdot \frac{1}{\|a\|}\Big\}.$$
If the condition
$$\frac{\theta}{1-\theta} < \frac{\sigma}{c}$$
is satisfied for some $0<\sigma<1$, then the matrix-valued Cauchy transform ${G}_a$ induces a mapping ${G}_a: \Omega\rightarrow\Omega'$ and we have the estimates
\begin{equation}\label{estimates}
\|{G}_a(b)\| \leq \|(b - a)^{-1}\| < \frac{\theta}{1-\theta}\cdot \frac{1}{\|a\|} \qquad\text{for all $b\in\Omega$}
\end{equation}
and
\begin{equation}\label{estimates-invers}
\|{G}_a(b)^{-1}\| < \frac{1}{1-\sigma}\cdot \|b\| \qquad\text{for all $b\in\Omega$}.
\end{equation}
\end{lemma}

\begin{proof}
For all $b\in\Omega$ we have
$$
{G}_a(b) - b^{-1} = E\big[(b-a)^{-1} - b^{-1}\big]= E\Big[\sum^\infty_{n=1} \big(b^{-1}a\big)^nb^{-1}\Big],
$$
which enables us to deduce
$$\|{G}_a(b) - b^{-1}\| \leq \|b^{-1}\| \sum^\infty_{n=1} \big(\|b^{-1}\|\|a\|\big)^n \leq \frac{\theta}{1-\theta}\cdot \|b^{-1}\| < \frac{\theta}{1-\theta}\cdot \frac{c}{\|b\|} < \sigma\cdot\frac{1}{\|b\|}.$$
Using Lemma \ref{norm-inverse}, this implies ${G}_a(b) \in \GL_m(\C)$ and \eqref{estimates-invers}. Since we already know from \eqref{estimates-0} in Lemma \ref{bounded}, that \eqref{estimates} holds, it follows ${G}_a(b)\in\Omega'$ and the proof is complete.
\end{proof}

\begin{remark}
Since domains of our holomorphic functions should be connected it is
necessary to note, that for $\kappa>0$ and $c>1$
$$\Omega = \big\{b\in\GL_m(\C)\mid \|b^{-1}\| < \kappa,\ 
\|b\|\cdot\|b^{-1}\| < c\big\}$$
and for $r>0$
$$\Omega' = \big\{b\in\GL_m(\C)\mid \|b\| < r\big\}$$
are pathwise connected subsets of $\M_m(\C)$. Indeed, if $b_1,b_2 \in \GL_m(\C)$ are given, we consider their polar decomposition $b_1=U_1P_1$ and $b_2=U_2P_2$ with unitary matrices $U_1,U_2\in\GL_m(\C)$ and positive-definite Hermitian matrices $P_1,P_2\in\GL_m(\C)$ and define (using functional calculus for normal elements in the $C^\ast$-algebra $\M_m(\C)$)
$$\gamma:\ [0,1] \rightarrow \GL_m(\C),\ t\mapsto U_1^{1-t} P_1^{1-t} U_2^t P_2^t.$$
Then $\gamma$ fulfills $\gamma(0)=b_1$ and $\gamma(1)=b_2$, and $\gamma([0,1])$ is contained in $\Omega$ and $\Omega'$ if $b_1,b_2$ are elements of $\Omega$ and $\Omega'$, respectively.
\end{remark}

Since the matrix-valued Cauchy transform is a solution of a special equation (cf. \cite{Speicher2}), we will be interested in the following situation: 

\begin{corollary}\label{unique}
Let $\eta: \GL_m(\C)\rightarrow\M_m(\C)$ be a holomorphic function satisfying
$$\|\eta(w)\| \leq M \|w\| \qquad\text{for all $w\in\GL_m(\C)$}$$
for some $M>0$. Moreover, we assume that
$$b{G}_a(b) = \1 + \eta({G}_a(b)){G}_a(b) \qquad\text{for all $b\in\Omega$}$$
holds. Let $0<\theta,\sigma<1$ and $c>1$ be given with
$$\frac{\theta}{1-\theta} < \sigma\min\Big\{\frac{1}{c},\ \frac{\|a\|^2}{M}\Big\}$$
and let $\Omega$ and $\Omega'$ be as in Lemma \ref{invertible}.

Then, for fixed $b\in\Omega$, the equation
\begin{equation}\label{F-equation}
bw = \1 + \eta(w)w,\qquad w\in\Omega'
\end{equation}
has a unique solution, which is given by $w={G}_a(b)$.
\end{corollary}

\begin{proof}
Let $b\in\Omega$ be given. For all $w\in\Omega'$ we get
$$\|\eta(w)\| \leq M \|w\| \leq \frac{\theta}{1-\theta}\cdot\frac{M}{\|a\|}$$
and therefore
$$\|b^{-1}\eta(w)\| \leq \|b^{-1}\| \|\eta(w)\| \leq \frac{\theta}{1-\theta}\cdot\frac{M}{\|a\|^2}\cdot\theta < \theta\sigma < 1.$$
This means, that $\1 - b^{-1}\eta(w)$ and hence $b-\eta(w)$ is invertible with
\begin{align*}
\|(b-\eta(w))^{-1}\| &\leq \|b^{-1}\|\|(1-b^{-1}\eta(w))^{-1}\|\\ &\leq \|b^{-1}\| \sum^\infty_{n=0} \|b^{-1} \eta(w)\|^n\\ &< \frac{\theta}{1-\theta\sigma}\cdot\frac{1}{\|a\|},
\end{align*}
and shows, that we have a well-defined and holomorphic mapping
$$\F:\ \Omega'\rightarrow\M_m(\C),\ w\mapsto (b-\eta(w))^{-1}$$
with
$$\|\F(w)\| = \|(b-\eta(w))^{-1}\| < \frac{\theta}{1-\theta\sigma}\cdot\frac{1}{\|a\|} < \frac{\theta}{1-\theta}\cdot\frac{1}{\|a\|}$$
and therefore $\F(w)\in\Omega'$.
 
Now, we want to show that $\F(\Omega')$ lies strictly inside $\Omega'$. We put
$$\epsilon := \min\Big\{\frac12\cdot\frac{1}{\|b\|+\sigma\|a\|},\ \Big(1-\frac{1-\theta}{1-\theta\sigma}\Big)\cdot\frac{\theta}{1-\theta}\cdot\frac{1}{\|a\|}\Big\} > 0$$
and consider $w\in\Omega'$ and $u\in\M_m(\C)$ with $\|u-\F(w)\| < \epsilon$. At first, we get
$$\|b-\eta(w)\| \leq \|b\| + \|\eta(w)\| \leq \|b\| + \frac{M}{\|a\|}\cdot \frac{\theta}{1-\theta} \leq \|b\| + \sigma\|a\|$$
and thus
$$\|u-(b-\eta(w))^{-1}\| = \|u-\F(w)\| < \epsilon \leq \frac12\cdot\frac{1}{\|b\| + \sigma\|a\|} \leq \frac12\cdot\frac{1}{\|b-\eta(w)\|},$$
which shows $u\in\GL_m(\C)$, and secondly
\begin{eqnarray*}
\|u\| &=& \|u-(b-\eta(w))^{-1}\| + \|\F(w)\|\\
      &<& \epsilon + \frac{1-\theta}{1-\theta\sigma}\cdot\frac{\theta}{1-\theta}\cdot\frac{1}{\|a\|}\\
      &<& \frac{\theta}{1-\theta}\cdot\frac{1}{\|a\|}
\end{eqnarray*}
which shows $u\in\Omega'$. 

Let now $w\in\Omega'$ be a solution of \eqref{F-equation}. This implies that
$$w^{-1}\F(w) = w^{-1}(b-\eta(w))^{-1} = \big(bw-\eta(w)w\big)^{-1} = \1,$$
and hence $\F(w)=w$. Since $\F: \Omega'\rightarrow\Omega'$ is holomorphic on the domain $\Omega'$ and $\F(\Omega')$ lies strictly inside $\Omega'$, it follows by the Theorem of Earle-Hamilton, Theorem \ref{Earle-Hamilton}, that $\F$ has exactly one fixed point. Because ${G}_a(b)$ (which is an element of $\Omega'$ by Lemma \ref{invertible}) solves \eqref{F-equation} by assumption and hence is already a fixed point of $\F$, it follows $w={G}_a(b)$ and we are done.
\end{proof}

\begin{remark}
Let $(\A',E')$ be an arbitrary operator-valued $C^\ast$-probability space with conditional expectation $E': \A'\rightarrow\M_m(\C)$. This provides us with a unital (and continuous) $\ast$-embedding $\iota: \M_m(\C) \rightarrow\A'$. In this section, we only considered the special embedding
$$\iota:\ \M_m(\C)\rightarrow\A,\ b\mapsto b\otimes\1,$$
which is given by the special structure $\A = \M_m(\C) \otimes \c$. But we can define matrix-valued resolvent sets, spectra and Cauchy transforms also in this more general framework. To be more precise, we put for all $a\in\A'$
$$\rho_m(a) := \{b\in \M_m(\C)\mid \text{$\iota(b) - a$ is invertible in $\A'$}\}$$
and $\sigma_m(a) := \M_m(\C) \backslash \rho_m(a)$ and
$${G}_a:\ \rho_m(a) \rightarrow \M_m(\C),\ b\mapsto E'\big[(\iota(b) - a)^{-1}\big].$$
We note, that all the results of this section stay valid in this general situation. 
\end{remark}

\section{Multivariate free central limit theorem}

\subsection{Setting and first observations}

Let $(X_i)_{i\in\N}$ be a sequence in the operator-valued probability space $(\A,E)$ with $\A=\M_m(\c)=\M_m(\C)\otimes\c$ and $E=\id\otimes\tau$, as defined in the previous section. We assume:
\begin{itemize}
 \item All $X_i$'s have the same $\ast$-distribution with respect to $E$ and their first moments vanish, i.e. $E[X_i]=0$.
 \item The $X_i$ are $*$-free with amalgamation over $\M_m(\C)$ (which means that the $\ast$-algebras $\X_i$, generated by $\M_m(\C)$ and $X_i$, are free with respect to $E$).
 \item We have $\displaystyle{\sup_{i\in\N} \|X_i\| < \infty}$.
\end{itemize}
If we define the linear (and hence holomorphic) mapping
$$\eta:\ \M_m(\C)\rightarrow\M_m(\C),\ b\mapsto E[X_ibX_i],$$
we easily get from the continuity of $E$, that
$$\|\eta(b)\| \leq \Big(\sup_{i\in\N} \|X_i\|\Big)^2 \|b\| \qquad\text{for all $b\in\M_m(\C)$}$$
holds. Hence we can find $M>0$ such that $\|\eta(b)\| < M\|b\|$ holds for all $b\in\M_m(\C)$. Moreover, we have for all $k\in\N$ and all $b_1,\dots,b_k\in\M_m(\C)$
$$\sup_{i\in\N} \|E[X_ib_1X_i\dots b_kX_i]\| \leq  \Big(\sup_{i\in\N} \|X_i\|\Big)^{k+1} \|b_1\|\cdots\|b_k\|.$$
Since $(X_i)_{i\in\N}$ is a sequence of centered free non-commutative random variables, Theorem 8.4 in \cite{Voiculescu} tells us that the sequence $(S_n)_{n\in\N}$ defined by
$$S_n := \frac{1}{\sqrt{n}} \sum^n_{i=1} X_i,\qquad n\in\N$$
converges to an operator-valued semicircular element $s$. Moreover, we know from Theorem 4.2.4 in \cite{Speicher2} that the operator-valued Cauchy transform ${G}_s$ satisfies
$$b{G}_s(b) = \1 + \eta({G}_s(b)){G}_s(b)\qquad\text{for all $b\in U_r$},$$
where we put $U_r := \{b\in\GL_m(\C)\mid \|b^{-1}\| < r\} \subseteq \rho_m(s)$ for all suitably small $r>0$.

By Proposition 7.1 in \cite{Junge}, the boundedness of the sequence $(X_i)_{i\in\N}$ guarantees boundedness of $(S_n)_{n\in\N}$ as well. In order to get estimates for the difference between the Cauchy transforms ${G}_s$ and ${G}_{S_n}$ we will also need the fact, that $(S_n)_{n\in\N}$ is bounded away from $0$. The precise statement is part of the following lemma, which also includes a similar statement for
$$S_n^{[i]} := S_n - \frac{1}{\sqrt{n}} X_i 
=\frac{1}{\sqrt{n}} \sum^n_{j=1\atop j\not=i} X_j
\qquad\text{for all $n\in\N$ and $1\leq i \leq n$}.$$

\begin{lemma}\label{lower-bound}
In the situation described above, we have for all $n\in\N$ and all $1\leq i \leq n$
$$\|S_n\| \geq \|\alpha\|^{\frac{1}{2}} \qquad\text{and}\qquad \|S_n^{[i]}\| \geq \sqrt{1-\frac{1}{n}} \|\alpha\|^{\frac{1}{2}},$$
where $\alpha := E[X_i^\ast X_i]\in\M_m(\C)$.
\end{lemma}

\begin{proof}
By the $*$-freeness of $X_1,X_2,\dots$, we have 
$$E[X_i^*X_j]=E[X_i^*]\cdot
E[X_j]=0,\qquad\text{for $i\not=j$}$$ 
and thus
$$\|S_n\|^2 = \|S_n^\ast S_n\| \geq \|E[S_n^\ast S_n]\| = \frac{1}{n} \bigg\|\sum^n_{i,j=1} E[X_i^\ast X_j] \bigg\| = \|\alpha\|.$$         
Similarly
\begin{align*}
\|S_n^{[i]}\|^2 &= \|(S_n^{[i]})^\ast S_n^{[i]}\|\\
               & \geq \|E[(S_n^{[i]})^\ast S_n^{[i]}]\|\\
                &= \bigg\|E[S_n^\ast S_n] - \frac{1}{n} E[X_i^\ast X_i]\bigg\|\\
                &= \frac{n-1}{n} \|\alpha\|,
\end{align*}
which proves the statement.
\end{proof}

We define for $n\in\N$
$$R_n:\ \rho_m(S_n) \rightarrow \A,\ b\mapsto \big(b - S_n\big)^{-1}$$
and for $n\in\N$ and $1\leq i \leq n$
$$R_n^{[i]}:\ \rho_m(S_n^{[i]}) \rightarrow \A,\ b\mapsto   \big(b - S_n^{[i]}\big)^{-1}.$$

\begin{lemma}
For all $n\in\N$ and $1\leq i \leq n$ we have
\begin{equation}\label{resolvent1}
R_n(b) = R_n^{[i]}(b) + \frac{1}{\sqrt{n}} R_n^{[i]}(b)X_iR_n^{[i]}(b) + \frac{1}{n} R_n(b) X_i R_n^{[i]}(b) X_i R_n^{[i]}(b)
\end{equation}
and
\begin{equation}\label{resolvent2}
R_n(b) = R_n^{[i]}(b) + \frac{1}{\sqrt{n}} R_n^{[i]}(b) X_i R_n(b)
\end{equation}
for all $b\in\rho_m(S_n) \cap \rho_m(S_n^{[i]})$.
\end{lemma}

\begin{proof}
We have
\begin{align*}
\big(b-S_n\big) R_n(b) \big(b-S_n^{[i]}\big)
&=\ b-S_n^{[i]}\\
 & =\ \big(b-S_n\big) + \frac{1}{\sqrt{n}} \big(b-S^{[i]}_n\big) R_n^{[i]}(b) X_i\\
&  =\ \big(b-S_n\big) + \frac{1}{\sqrt{n}} \big(b-S_n\big) R_n^{[i]}(b) X_i + \frac{1}{n} X_i R_n^{[i]}(b) X_i,
\end{align*}
which leads, by multiplication with $R_n(b)=(b-S_n)^{-1}$ from the left and with $R_n^{[i]}(b)= (b-S_n^{[i]})^{-1}$ from the right, to equation \eqref{resolvent1}.

Moreover, we have
$$\big(b - S^{[i]}_n\big) R_n(b) \big(b - S_n\big)
= b - S^{[i]}_n = \big(b - S_n\big) + \frac{1}{\sqrt{n}} X_i,$$
which leads, by multiplication with $R_n(b)=(b-S_n)^{-1}$ from the right and with $R_n^{[i]}(b)= (b-S_n^{[i]})^{-1}$ from the left, to equation \eqref{resolvent2}.
\end{proof}

Obviously, we have
$$G_n := {G}_{S_n} = E\circ R_n \qquad\text{and}\qquad G_n^{[i]} := {G}_{S_n^{[i]}} = E\circ R_n^{[i]}.$$

\subsection{Proof of the main theorem}

During this subsection, let $0<\theta,\sigma<1$ and $c>1$ be given, such that
\begin{equation}\label{constants1}
\frac{\theta}{1-\theta} < \sigma \min\Big\{\frac{1}{c},\ \frac{\|\alpha\|}{M}\Big\}
\end{equation}
holds. For all $n\in\N$ we define
$$\kappa_n := \theta \min\Big\{\frac{1}{\|s\|}, \frac{1}{\|S_n\|}, \frac{1}{\|S_n^{[1]}\|},\dots, \frac{1}{\|S_n^{[n]}\|}\Big\}$$
and
$$\Omega_n := \big\{b\in\GL_m(\C)\mid \|b^{-1 }\| < \kappa_n,\ \|b\|\cdot\|b^{-1}\| < c\big\}.$$
Lemma \ref{resolvent} shows, that $\Omega_n$ is a subset of $\rho_m(S_n)$.

\begin{theorem}\label{Theta}
For all $2\leq n\in\N$ the function $G_n$ satisfies the following equation
$$\Lambda_n(b) G_n(b) = \1 + \eta(G_n(b))G_n(b),\qquad b\in \Omega_n,$$
where
$$\Lambda_n:\ \Omega_n\rightarrow\M_m(\C),\ b\mapsto b-\Theta_n(b) G_n(b)^{-1},$$
with a holomorphic function
$$\Theta_n:\ \Omega_n\rightarrow\M_m(\C)$$
satisfying
$$\sup_{b\in\Omega_n} \|\Theta_n(b)\| \leq \frac{C}{\sqrt{n}}$$
with a constant $C>0$, independent of $n$.
\end{theorem}

\begin{proof}
(i) Let $n\in\N$ and $b\in\rho_m(S_n)$ be given. Then we have
$$S_nR_n(b) = bR_n(b) - (b - S_n)R_n(b) = bR_n(b) - \1$$
and hence
$$E[S_n R_n(b)] = E\big[bR_n(b) - \1\big] = b G_n(b) - \1.$$
(ii) Let $n\in\N$ be given. For all
$$b\in\rho_{m,n} := \rho_m(S_n) \cap \bigcap_{i=1}^n \rho_m(S_n^{[i]})$$
we deduce from the formula in \eqref{resolvent1}, that
\begin{align*}
&E[S_nR_n(b)] 
= \frac{1}{\sqrt{n}} \sum^n_{i=1} E[X_i R_n(b)]\\
&= \frac{1}{\sqrt{n}} \sum^n_{i=1} \bigg(E\big[X_iR_n^{[i]}(b)\big] + \frac{1}{\sqrt{n}} E\big[X_i R_n^{[i]}(b) X_i R_n^{[i]}(b)\big]\\
&  \qquad\qquad\qquad\qquad + \frac{1}{n} E\big[X_i R_n(b) X_i R_n^{[i]}(b) X_i R_n^{[i]}(b)\big]\bigg)\\
&= \frac{1}{n} \sum^n_{i=1} \bigg(E\big[X_i R_n^{[i]}(b) X_i R_n^{[i]}(b)\big] + \frac{1}{\sqrt{n}} E\big[X_i R_n(b) X_i R_n^{[i]}(b) X_i R_n^{[i]}(b)\big]\bigg)\\
&= \frac{1}{n} \sum^n_{i=1} \bigg(E\big[X_i G_n^{[i]}(b) X_i\big] G_n^{[i]}(b) + \frac{1}{\sqrt{n}} E\big[X_i R_n(b) X_i R_n^{[i]}(b) X_i R_n^{[i]}(b)\big]\bigg)\\
&= \frac{1}{n} \sum^n_{i=1} \Big(\eta(G_n^{[i]}(b)) G_n^{[i]}(b) + r_{n,1}^{[i]}(b)\Big),
\end{align*}
where
$$r_{n,1}^{[i]}:\ \rho_m(S_n) \cap \rho_m(S_n^{[i]}) \rightarrow \M_m(\C),\ b\mapsto \frac{1}{\sqrt{n}} E\big[X_i R_n(b) X_i R_n^{[i]}(b) X_i R_n^{[i]}(b)\big].$$
There we used the fact, that, since the $(X_j)_{j\in\N}$ are free with respect to $E$, also $X_i$ is free from $R_n^{[i]}$, and thus we have
$$E\big[X_iR_n^{[i]}(b)\big] = E[X_i] E\big[R_n^{[i]}(b)\big] = 0$$
and
$$E\big[X_i R_n^{[i]}(b) X_i R_n^{[i]}(b)\big] = E\big[X_i E\big[R_n^{[i]}(b)\big] X_i\big] E \big[R_n^{[i]}(b)\big].$$
(iii) Taking \eqref{resolvent2} into account, we get for all $n\in\N$ and $1\leq i \leq n$
$$G_n(b) = E\big[R_n(b)\big] = E\big[R_n^{[i]}(b)\big] + \frac{1}{\sqrt{n}} E\big[R_n^{[i]}(b) X_i R_n(b)\big] = G_n^{[i]}(b) - r_{n,2}^{[i]}(b)$$
and therefore
$$G_n^{[i]}(b) = G_n(b) + r_{n,2}^{[i]}(b)$$
for all $b\in\rho_m(S_n) \cap \rho_m(S_n^{[i]})$, where we put
$$r_{n,2}^{[i]}:\ \rho_m(S_n) \cap \rho_m(S_n^{[i]}) \rightarrow \M_m(\C),\ b\mapsto -\frac{1}{\sqrt{n}} E\big[R_n^{[i]}(b) X_i R_n(b)\big].$$
(iv) The formula in (iii) enables us to replace $G_n^{[i]}$ in (ii) by $G_n$. Indeed, we get
\begin{eqnarray*}
E[S_nR_n(b)] &=& \frac{1}{n} \sum^n_{i=1} \Big(\eta(G_n^{[i]}(b)) G_n^{[i]}(b) + r_{n,1}^{[i]}(b)\Big)\\
             &=& \frac{1}{n} \sum^n_{i=1} \Big(\eta\big(G_n(b) + r_{n,2}^{[i]}(b)\big) \big(G_n(b) + r_{n,2}^{[i]}(b)\big) + r_{n,1}^{[i]}(b)\Big)\\
             &=& \eta(G_n(b))G_n(b) + \frac{1}{n} \sum^n_{i=1} r_{n,3}^{[i]}(b)
\end{eqnarray*}
for all $b\in\rho_{m,n}$, where the function
$$r_{n,3}^{[i]}:\ \rho_m(S_n) \cap \rho_m(S_n^{[i]}) \rightarrow \M_m(\C)$$ 
is defined by
$$r_{n,3}^{[i]}(b) := \eta(G_n(b)) r_{n,2}^{[i]}(b) + \eta(r_{n,2}^{[i]}(b)) G_n(b) + \eta(r_{n,2}^{[i]}(b)) r_{n,2}^{[i]}(b) + r_{n,1}^{[i]}(b).$$
(v) Combining the results from (i) and (iv), it follows
$$b G_n(b) - 1 = E[S_nR_n(b)] = \eta(G_n(b))G_n(b) + \Theta_n(b),$$
where we define
$$\Theta_n:\ \rho_{m,n} \rightarrow\M_m(\C),\ b\mapsto \frac{1}{n} \sum^n_{i=1} r_{n,3}^{[i]}(b).$$
Due to \eqref{constants1}, Lemma \ref{resolvent} and Lemma \ref{invertible} show that $\Omega_n \subseteq \rho_{m,n}$ and $G_n(b)\in\GL_m(\C)$ for $b\in \Omega_n$. This gives
$$\big(b-\Theta_n(b) G_n(b)^{-1}\big) G_n(b) = \1 + \eta(G_n(b))G_n(b)$$
and hence, as desired, for all $b\in\Omega_n$
$$\Lambda_n(b) G_n(b) = \1 + \eta(G_n(b))G_n(b).$$
(v) The definition of $\Omega_n$ gives, by Lemma \ref{estimates} and by Lemma \ref{lower-bound}, the following estimates
$$\|G_n(b)\| \leq \|R_n(b)\| \leq \frac{\theta}{1-\theta}\cdot\frac{1}{\|S_n\|} \leq \frac{\theta}{1-\theta}\cdot\frac{1}{\|\alpha\|^{\frac{1}{2}}},\qquad b\in \Omega_n$$
and
$$\|G_n^{[i]}(b)\| \leq \|R_n^{[i]}(b)\| \leq \frac{\theta}{1-\theta}\cdot\frac{1}{\|S^{[i]}_n\|} \leq \frac{\theta}{1-\theta}\cdot\frac{1}{\sqrt{1-\frac{1}{n}}\|\alpha\|^{\frac{1}{2}}},\qquad b\in \Omega_n.$$
Therefore, we have for all $b\in\Omega_n$ by (ii)
$$\|r_{n,1}^{[i]}(b)\| \leq \frac{1}{\sqrt{n}} \|X_i\|^3 \|R_n(b)\| \|R_n^{[i]}(b)\|^2 \leq \frac{1}{\sqrt{n}} \frac{n}{n-1} \Big(\frac{\theta}{1-\theta}\frac{1}{\|\alpha\|^{\frac{1}{2}}}\Big)^3 \|X_i\|^3$$
and by (iii)
$$\|r_{n,2}^{[i]}(b)\| \leq \frac{1}{\sqrt{n}} \|X_i\| \|R_n(b)\| \|R_n^{[i]}(b)\| \leq \frac{1}{\sqrt{n-1}} \Big(\frac{\theta}{1-\theta}\frac{1}{\|\alpha\|^{\frac{1}{2}}}\Big)^2 \|X_i\| $$
and finally by (iv)
\begin{eqnarray*}
\|r_{n,3}^{[i]}(b)\| &\leq& 2M \|G_n(b)\| \|r_{n,2}^{[i]}(b)\| + M\|r_{n,2}^{[i]}(b)\|^2 + \|r_{n,1}^{[i]}(b)\|\\
                     &\leq& \frac{1}{\sqrt{n-1}} \Big(\frac{\theta}{1-\theta}\frac{1}{\|\alpha\|^{\frac{1}{2}}}\Big)^3 \|X_i\| \cdot\\
                     & & \bigg(2M+ \frac{1}{\sqrt{n-1}}M\Big(\frac{\theta}{1-\theta}\frac{1}{\|\alpha\|^{\frac{1}{2}}}\Big) \|X_i\| + \sqrt{\frac{n}{n-1}} \|X_i\|^2\bigg)\\
                     &\leq& \frac{C}{\sqrt{n}}
\end{eqnarray*}
for all $b\in\Omega_n$, where $C>0$ is a constant, which is independent of $n$. Hence, it follows from (v) that
$$\sup_{b\in\Omega_n} \|\Theta_n(b)\| \leq \frac{C}{\sqrt{n}}.$$
\end{proof}

The definition of $\Omega_n$ ensures, that
$$G := {G}_s:\ \rho_m(s) \rightarrow \M_m(\C)$$
satisfies
$$b G(b) = \1 + \eta(G(b)) G(b) \qquad\text{for all $b\in\Omega$},$$
where
$$\Omega := \Big\{b\in\GL_m(\C)\mid \|b^{-1}\| < \theta\cdot \frac{1}{\|s\|},\ \|b\|\cdot\|b^{-1}\| < c\Big\} \supseteq \Omega_n.$$

We choose
\begin{equation}\label{constants2}
0<\gamma< \frac{c-1}{c+1} \qquad\text{and}\qquad 0<\theta^\ast<(1-\gamma)\theta
\end{equation}
(note, that $0<\gamma<1$) and we put $c^\ast := c-(1+c)\gamma$, which fulfills clearly $1<c^\ast<c$. Since we have $\theta^\ast < \theta$ and $c^\ast < c$, we see
$$\frac{\theta^\ast}{1-\theta^\ast}c^\ast < \frac{\theta}{1-\theta}c < \sigma$$
and hence
\begin{equation}\label{constants3}
\frac{\theta^\ast}{1-\theta^\ast} < \frac{\sigma}{c^\ast}.
\end{equation}
Finally, we define
$$\kappa_n^\ast := \theta^\ast \min\Big\{\frac{1}{\|s\|}, \frac{1}{\|S_n\|}, \frac{1}{\|S_n^{[1]}\|},\dots, \frac{1}{\|S_n^{[n]}\|}\Big\}$$
and
$$\Omega_n^\ast := \Big\{b\in\GL_m(\C)\mid \|b^{-1}\| < \kappa^\ast_n,\ \|b\|\cdot\|b^{-1}\| < c^\ast\Big\} \subseteq \Omega_n.$$

\begin{corollary}
There exists $N\in\N$ such that
$$\Lambda_n(\Omega_n^\ast) \subseteq \Omega_n \qquad\text{for all $n\geq N$}.$$
\end{corollary}

\begin{proof}
Since we have by Theorem \ref{Theta}
$$\sup_{b\in\Omega_n} \|\Theta_n(b)\| \leq \frac{C}{\sqrt{n}}$$
for all $2\leq n\in\N$, we can choose an $N\in\N$ such that
$$\sup_{b\in\Omega_n} \|\Theta_n(b)\| \leq \frac{\gamma}{c^\ast}(1-\sigma)$$
holds for all $n\geq N$. Now, we get for all $b\in \Omega_n^\ast$:
\begin{itemize}
 \item[(i)] $\Lambda_n(b)$ is invertible: Since \eqref{estimates-invers} gives $$\|G_n(b)^{-1}\| \leq \frac{1}{1-\sigma} \|b\| \qquad\text{for all $b\in\Omega_n$},$$ we immediately get $$\|\Lambda_n(b)-b\| \leq \|\Theta_n(b)\|\|G_n(b)^{-1}\| < \gamma \frac{\|b\|}{c^\ast} < \gamma \frac{1}{\|b^{-1}\|} < \frac{1}{\|b^{-1}\|}$$
 \item[(ii)] We have $\|\Lambda_n(b)^{-1}\| < \kappa_n$: Using Lemma \ref{norm-inverse}, we get from (i) that $$\|\Lambda_n(b)^{-1}\| \leq \frac{1}{1-\gamma}\|b^{-1}\| < \frac{\kappa^\ast_n}{1-\gamma} < \kappa_n.$$
 \item[(iii)] We have $\|\Lambda_n(b)\|\|\Lambda_n(b)^{-1}\| < c$: Using $$\|\Lambda_n(b)-b\| < \gamma \frac{\|b\|}{c^\ast}$$ from (i) and $$\|\Lambda_n(b)^{-1}\| < \frac{1}{1-\gamma}\|b^{-1}\|$$ from (ii), we get \begin{eqnarray*}\|\Lambda_n(b)\|\|\Lambda_n(b)^{-1}\| &\leq& \big(\|b\| + \|\Lambda_n(b)-b\|\big) \|\Lambda_n(b)^{-1}\|\\ &<& \Big(1+\frac{\gamma}{c^\ast}\Big)\frac{1}{1-\gamma}\cdot\|b\|\|b^{-1}\|\\ &<& \frac{c^\ast+\gamma}{1-\gamma} < c.\end{eqnarray*}
\end{itemize}
Finally, this shows $\Lambda_n(b) \in \Omega_n$.
\end{proof}

\begin{corollary}
For all $n\geq N$ we have
$$G_n(b) = G(\Lambda_n(b)) \qquad\text{for all $b\in \Omega_n^\ast$}.$$
\end{corollary}

\begin{proof}
For all $n\in\N$ we define
$$\Omega_n' := \Big\{b\in\GL_m(\C)\mid \|b\| < \frac{\kappa_n}{1-\theta}\Big\}.$$
Let $n\geq N$ and $b\in\Omega_n^\ast$ be given. We know, that
$$\Lambda_n(b) G(\Lambda_n(b)) = \1 + \eta(G(\Lambda_n(b))) G(\Lambda_n(b))$$
holds, i.e. $w=G(\Lambda_n(b)) \in \Omega_n'$ is a solution of the equation
$$\Lambda_n(b) w = \1 + \eta(w) w,\qquad w\in\Omega_n'.$$
Combining \eqref{constants1} with Lemma \ref{lower-bound}, we get
$$\frac{\theta}{1-\theta} < \sigma\min\Big\{\frac{1}{c},\ \frac{\|\alpha\|}{M}\Big\} \leq \sigma\min\Big\{\frac{1}{c},\ \frac{\|S_n\|^2}{M};\ n\in\N\Big\}.$$
Hence, the equation above has, by Theorem \ref{unique}, the unique solution $w=G_n(b) \in\Omega_n'$. This implies, as desired, $G_n(b) = G(\Lambda_n(b))$.
\end{proof}

\begin{corollary}
For all $n\geq N$ we have
$$\|G(b) - G_n(b)\| \leq C' \frac{1}{\sqrt{n}}\|b\| \qquad\text{for all $b\in \Omega_n^\ast$},$$
where $C'>0$ is a constant independent of $n$.
\end{corollary}

\begin{proof}
For all $b\in \Omega_n^\ast \subseteq \Omega_n \subseteq \Omega$ we have
\begin{eqnarray*}
G(b) - G_n(b) &=& G(b) - G(\Lambda_n(b))\\
              &=& E\big[(b - s)^{-1} - (\Lambda_n(b) - s)^{-1}\big]\\
              &=& E\big[(b - s)^{-1} (\Lambda_n(b) - b) (\Lambda_n(b)- s)^{-1}\big]
\end{eqnarray*}
and therefore by \eqref{estimates-invers}, which gives
$$\|G_n(b)^{-1}\| \leq \frac{1}{1-\sigma} \|b\| \qquad\text{for all $b\in\Omega_n^\ast$},$$
and (since $\Lambda_n(b)\in\Omega_n\subseteq \Omega$) by \eqref{estimates}
\begin{eqnarray*}
\|G(b) - G_n(b)\| &\leq& \|(b-s)^{-1}\|\cdot \|\Lambda_n(b)-b\|\cdot \|(\Lambda_n(b)-s)^{-1}\|\\
                  &\leq& \Big(\frac{\theta}{1-\theta} \cdot \frac{1}{\|s\|}\Big)^2 \cdot\|\Theta_n(b)\|\cdot\|G_n(b)^{-1}\|\\
                  &\leq& C' \frac{1}{\sqrt{n}} \|b\|,
\end{eqnarray*}
where
$$C' := \frac{C}{1-\sigma} \Big(\frac{\theta}{1-\theta} \cdot \frac{1}{\|s\|}\Big)^2 > 0.$$
This proves the corollary.
\end{proof}

We recall, that the sequence $(X_i)_{i\in\N}$ is bounded, which implies boundedness of the sequence $(S_n)_ {n\in\N}$ as well. This has the important consequence, that
$$\kappa_n^\ast  = \theta^\ast \min\Big\{\frac{1}{\|s\|}, \frac{1}{\|S_n\|}, \frac{1}{\|S_n^{[1]}\|},\dots, \frac{1}{\|S_n^{[n]}\|}\Big\} \geq \kappa^\ast$$
for some $\kappa^\ast > 0$. If we define
$$\Omega^\ast := \Big\{b\in\GL_m(\C)\mid \|b^{-1}\| < \kappa^\ast,\ \|b\|\cdot\|b^{-1}\| < c^\ast\Big\},$$
we easily see $\Omega^\ast \subseteq \Omega_n^\ast$ for all $n\in\N$. Hence, by renaming $\Omega^*$ to $\Omega$ etc., we have shown our main Theorem \ref{Berry-Esseen}.\\

We conclude this section with the following remark about the geometric structure of subsets of $\M_m(\C)$ like $\Omega$.

\begin{lemma}\label{annulus}
For $\kappa>0$ and $c>1$ we consider
$$\Omega := \Big\{b\in\GL_m(\C)\mid \|b^{-1}\| < \kappa,\ \|b\|\cdot\|b^{-1}\| < c\Big\}.$$
For $\lambda,\mu\in\C\backslash\{0\}$ we define
$$\Lambda(\lambda,\mu) := \begin{pmatrix} \lambda & 0      & \hdots & 0\\
                                          0       & \mu    & \hdots & 0\\
                                          \vdots  & \vdots & \ddots & \vdots\\
                                          0       & 0      & \hdots & \mu
                                          \end{pmatrix} \in\GL_m(\C).$$
If $\frac{1}{\kappa}<|\mu|$ holds, we have $\Lambda(\lambda,\mu) \in \Omega$ for all
\begin{equation}\label{annulus-condition}
\max\Big\{\frac{1}{\kappa}, \frac{|\mu|}{c}\Big\} < |\lambda| < c|\mu|.
\end{equation}
Particularly, we have for all $|\lambda|>\frac{1}{\kappa}$, that $\lambda\1 \in\Omega$.
\end{lemma}

\begin{proof}
Let $\mu\in\C\backslash\{0\}$ with $\frac{1}{\kappa}<|\mu|$ be given. For all $\lambda\in\C\backslash\{0\}$, which satisfy \eqref{annulus-condition}, we get
$$\|\Lambda(\lambda,\mu)^{-1}\| = \|\Lambda(\lambda^{-1},\mu^{-1})\| = \max\big\{|\lambda|^{-1},|\mu|^{-1}\big\} < \kappa.$$
and
\begin{eqnarray*}
\|\Lambda(\lambda,\mu)\|\cdot \|\Lambda(\lambda,\mu)^{-1}\|
&=& \max\big\{|\lambda|,|\mu|\big\}\cdot \max\big\{|\lambda|^{-1},|\mu|^{-1}\big\}\\
&=& \begin{cases}|\mu| |\lambda|^{-1}, & \text{if $|\lambda| < |\mu|$}\\ |\lambda||\mu|^{-1}, & \text{if $|\lambda| \geq |\mu|$}\end{cases}\\
&<& c,
\end{eqnarray*}
which implies $\Lambda(\lambda,\mu)\in\Omega$. In particular, for $\lambda\in\C\backslash\{0\}$ with $|\lambda|>\frac{1}{\kappa}$ we see that $\mu=\lambda$ fulfills \eqref{annulus-condition} and it follows $\lambda\1 = \Lambda(\lambda,\lambda) \in \Omega$.
\end{proof}

\subsection{Application to multivariate situation}

\subsubsection{Multivariate free central limit theorem}

Let $(x_i^{(k)})_{k=1}^d$, $i\in\N$, be free and identically distributed sets of $d$ self-adjoint non-zero random variables in some non-commutative $C^\ast$-probability space $(\c,\tau)$, with $\tau$ faithful, such that
$$\tau(x_i^{(k)}) = 0 \qquad\text{for $k=1,\dots,d$ and all $i\in\N$}$$
and
\begin{equation}\label{multi-bounded}
\sup_{i\in\N} \max_{k=1,\dots,d} \|x_i^{(k)}\| < \infty.
\end{equation}
We denote by $\Sigma =(\sigma_{k,l})_{k,l=1}^d$, where $\sigma_{k,l} := \tau(x_i^{(k)} x_i^{(l)})$, their joint covariance matrix. Moreover, we put
$$S_n^{(k)} := \frac{1}{\sqrt{n}} \sum^n_{i=1} x_i^{(k)} \qquad\text{for $k=1,\dots,d$ and all $n\in\N$}.$$
We know (cf. \cite{Speicher1}), that $(S_n^{(1)},\dots,S_n^{(d)})$ converges in distribution as $n\rightarrow\infty$ to a semicircular family $(s_1,\dots,s_d)$ of covariance $\Sigma$. For notational convenience we will assume that $s_1,\dots,s_d$ live also in $(\c,\tau)$; this can always be achieved by enlarging
$(\c,\tau)$.

Using Proposition 2.1 and Proposition 2.3 in \cite{Haagerup1}, for each polynomial $p$ of degree $g$ in $d$ non-commuting variables, we can find $m\in\N$ and $a_1,\dots,a_d \in  \M_m(\C)$ such that 
$$\lambda\1 - p(S_n^{(1)},\dots,S_n^{(d)})\qquad\text{and}\qquad \lambda\1 - p(s_1,\dots,s_d)$$ 
are invertible in $\c$ if and only if 
$$\Lambda(\lambda,1) - S_n\qquad \text{and} \qquad \Lambda(\lambda,1) - s,$$
respectively, 
are invertible in $\A=\M_m(\c)$. The matrices $\Lambda(\lambda,1)\in
\M_m(\C)$ were defined in Lemma \ref{annulus}, and $S_n$ and $s$ are defined as follows:
$$S_n := \sum^d_{k=1} a_k \otimes S^{(k)}_n \in\A \qquad\text{for all $n\in\N$}$$
and
$$s := \sum^d_{k=1} a_k \otimes s_k \in\A.$$
If we also put
$$X_i := \sum^d_{k=1} a_k \otimes x_i^{(k)} \in\A \qquad\text{for all $i\in\N$},$$
then we have
$$S_n = \frac{1}{\sqrt{n}} \sum^n_{i=1} X_i.$$
We note, that the sequence $(X_i)_{i\in\N}$ is $\ast$-free with respect to the conditional expectation $E:\A=\M_m(\c)\to\M_m(\C)$ and that all the $X_i$'s have the same $\ast$-distribution with respect to $E$ and that they satisfy $E[X_i]=0$. In addition, \eqref{multi-bounded} implies $\sup_{i\in\N} \|X_i\| < \infty$. Hence, the conditions of Theorem \ref{Berry-Esseen} are fulfilled. But before we apply it, we note that $(S_n)_{n\in\N}$ converges in distribution
(with respect to $E$) to $s$, which is an $\M_m(\C)$-valued semicircular element with covariance mapping
$$\eta:\ \M_m(\C) \rightarrow \M_m(\C),\ b\mapsto E[sbs],$$
which is given by
$$\eta(b) = E[sbs] = \sum^d_{k,l=1} \id\otimes\tau[(a_k\otimes s_k)(b\otimes\1)(a_l\otimes s_l)] = \sum^d_{k,l=1} a_kba_l \sigma_{k,l}.$$
Now, we get from Theorem \ref{Berry-Esseen} constants $\kappa^\ast>0$, $c^\ast>0$ and $C'>0$ and $N\in\N$ such that we have for the difference of the operator-valued Cauchy transforms
$$G_s(b):=E[(b-s)^{-1}]\qquad\text{and}\qquad G_{S_n}(b):=E[(b-S_n)^{-1}]$$
the estimate
$$\|{G}_{s}(b) - {G}_{S_n}(b)\| \leq C' \frac{1}{\sqrt{n}}\|b\| \qquad\text{for all $b\in \Omega^\ast$ and $n\geq N$},$$
where we put
$$\Omega^\ast := \Big\{b\in\GL_m(\C)\mid \|b^{-1}\| < \kappa^\ast,\ \|b\|\cdot\|b^{-1}\| < c^\ast\Big\}.$$

Moreover, Proposition 2.3 in \cite{Haagerup1} tells us
$$\big(\lambda\1 - p(S_n^{(1)},\dots,S_n^{(d)})\big)^{-1} = (\pi\otimes\id_\c)\big((\Lambda(\lambda,1) - S_n)^{-1}\big)$$
and
$$\big(\lambda\1 - p(s_1,\dots,s_d)\big)^{-1} = (\pi\otimes\id_{\c'})\big((\Lambda(\lambda,1) - s)^{-1}\big),$$
where $\pi: \M_m(\C)\rightarrow\C$ is the mapping given by $\pi((a_{i,j})_{i,j=1,\dots,m}) := a_{1,1}$. Since
$\tau\circ(\pi\otimes\id_\c) = \pi \circ E$,
this implies a direct connection between the operator-valued Cauchy transforms of $S_n$ and $s$ and the scalar-valued Cauchy transforms of $P_n:=p(S_n^{(1)},\dots,S_n^{(d)})$ and $P:=p(s_1,\dots,s_d)$, respectively. To be more precise, we get
$$G_{P_n}(\lambda) :=\tau[(\lambda-P_n)^{-1}]=
 \pi\big({G}_{S_n}(\Lambda(\lambda,1))\big)$$
and
$$ G_P(\lambda):=\tau[(\lambda-P)^{-1}] = \pi\big({G}_s(\Lambda(\lambda,1))\big)$$
for all $\lambda\in\rho_\c(P_n)$ and $\lambda\in\rho_{\c}(P)$, respectively.

If we choose $\mu\in\C$ such that $|\mu|>\frac{1}{\kappa^\ast}$ holds, it follows from Lemma \ref{annulus}, that $\Lambda(\lambda,\mu) \in \Omega^\ast$ is fulfilled for all $\lambda\in A(\mu)$, where $A(\mu)\subseteq\C$ denotes the open set of all $\lambda\in\C$ satisfying \eqref{annulus-condition}, i.e.
$$A(\mu) := \Big\{\lambda\in\C\mid \max\Big\{\frac{1}{\kappa^\ast}, \frac{|\mu|}{c^\ast}\Big\} < |\lambda| < c^\ast|\mu|\Big\}.$$

If we apply Proposition 2.1 and Proposition 2.2 in \cite{Haagerup1} to the polynomial $\frac{1}{\mu^g} p$ (which corresponds to the operators $\frac{1}{\mu} S_n$ and $\frac{1}{\mu} S$), we easily deduce that $$\lambda\1 - \frac{1}{\mu^{g-1}}p(S_n^{(1)},\dots,S_n^{(d)})\qquad\text{and}\qquad \lambda\1 - \frac{1}{\mu^{g-1}}p(s_1,\dots,s_d)$$ are invertible in $\c$ if and only if $$\Lambda(\lambda,\mu) - S_n\qquad\text{and}\qquad \Lambda(\lambda,\mu) - S_n,$$
respectively, 
are invertible in $\A$. Moreover, we have
$$\mu^{g-1} G_{P_n}(\lambda\mu^{g-1}) = \pi\big({G}_{S_n}(\Lambda(\lambda,\mu))\big)$$
and
$$ \mu^{g-1} G_P(\lambda\mu^{g-1}) = \pi\big({G}_s(\Lambda(\lambda,\mu))\big)$$
for all $\lambda\in\rho_\c(\frac{1}{\mu^{g-1}} P_n)$ and $\lambda\in\rho_\c(\frac{1}{\mu^{g-1}} P)$, respectively.

Particularly, for all $\lambda\in A(\mu)$ we get $\Lambda(\lambda,\mu)\in\Omega^\ast$ and hence $\lambda\in\rho_\c(\frac{1}{\mu^{g-1}} P_n) \cap \rho_\c(\frac{1}{\mu^{g-1}} P)$ for all $n\geq N$. Therefore, Theorem \ref{Berry-Esseen} implies
\begin{eqnarray*}
|\mu|^{g-1}|G_P(\lambda\mu^{g-1}) - G_{P_n}(\lambda\mu^{g-1})|
&=& \big|\pi\big({G}_s(\Lambda(\lambda,\mu)) - {G}_{S_n}(\Lambda(\lambda,\mu))\big)\big|\\
&\leq& \big\|{G}_s(\Lambda(\lambda,\mu)) - {G}_{S_n}(\Lambda(\lambda,\mu))\big\|\\
&\leq& C' \frac{1}{\sqrt{n}} \|\Lambda(\lambda,\mu)\|\\
&\leq& C' \frac{1}{\sqrt{n}} \max\{|\lambda|,|\mu|\}\\
&\leq& C'c^\ast |\lambda| \frac{1}{\sqrt{n}}
\end{eqnarray*}
and hence
$$|G_P(\lambda\mu^{g-1}) - G_{P_n}(\lambda\mu^{g-1})| \leq  C'c^\ast \frac{1}{\sqrt{n}} |\lambda \mu^{g-1}|.$$
This means, that
$$|G_P(z) - G_{P_n}(z)| \leq C'c^\ast \frac{1}{\sqrt{n}} |z|$$
holds for all $z\in\C$ with $\frac{z}{\mu^{g-1}} \in A(\mu)$ and all $n\geq N$. By definition of $A(\mu)$, we particularly get
$$|G_P(z) - G_{P_n}(z)| \leq C \frac{1}{\sqrt{n}} \qquad\text{for all $\frac{1}{c^\ast}|\mu|^g < |z| < c^\ast|\mu|^g$ and $n\geq N$},$$
where we put $C := C'(c^\ast)^2|\mu|^g > 0$. Since $z\mapsto G_P(z) - G_{P_n}(z)$ is holomorphic on $\{z\in\C\mid |z| > R\}$ for $R:=\frac{1}{c^\ast}|\mu|^g>0$ and extends holomorphically to $\infty$, the maximum modulus principle gives
$$|G_P(z) - G_{P_n}(z)| \leq C \frac{1}{\sqrt{n}} \qquad\text{for all $|z|>R$ and $n\geq N$}.$$
This shows Theorem \ref{Berry-Esseen-multivariate}.

\subsubsection{Estimates in terms of the Kolmogorov distance}

In the classical case, estimates between scalar-valued Cauchy transfoms can be established (for self-adjoint operators) in all of the upper complex plane and lead then to estimates in terms of the Kolmogorov distance. In the case treated above, we have a statement about the behavior of the difference between two Cauchy transforms only near infinity. Even in the case, where our operators are self-adjoint, we still have to transport estimates from infinity to the real line, and hence we can not apply the results of Bai \cite{Bai} directly. A partial solution to this problem was given in the appendix of \cite{Speicher4} with the following theorem, formulated in terms of probability measures instead of operators. There we use the notation $G_\mu$ for the Cauchy
transform of the measure  $\mu$, and put
$$D^+_R := \{z\in\C\mid \Im(z)>0,\ |z|>R\}.$$

\begin{theorem}\label{thm:4.8}
Let $\mu$ be a probability measure with compact support contained in an interval $[-A,A]$ such that the cumulative distribution function $\F_\mu$ satisfies
\begin{equation*}
|\F_\mu(x+t) - \F_\mu(x)| \leq \rho |t| \qquad\text{for all $x,t\in\R$}
\end{equation*}
for some constant $\rho>0$. Then for all $R>0$ and $\beta\in(0,1)$ we can find $\Theta>0$ and $m_0>0$ such that for any probability measure $\nu$ with compact support contained in $[-A,A]$, which satisfies
$$\sup_{z\in D_R^+} |G_\mu(z)-G_\nu(z)| \leq e^{-m}$$
for some $m>m_0$, the Kolmogorov distance $\displaystyle{\Delta(\mu,\nu):=\sup_{x\in\R} |\F_\mu(x)-\F_\nu(x)|}$ fulfills
$$\Delta(\mu,\nu)\leq \Theta\frac{1}{m^\beta}.$$
\end{theorem}

Obviously, this leads to the following questions: First, the stated estimate for the speed of convergence in terms of the Kolmogorov distance is far from the expected one. We hope to improve this result in a future work. Furthermore, in order to apply this theorem, we have to ensure that $p(s_1,\dots,s_d)$ has a continuous density. As mentioned in the introduction, it is a still unsolved problem, whether this is always true for any self-adjoint polynomials $p$. 

\section*{Acknowledgment}
This project was initiated by disussions with Friedrich G\"otze during
the visit of the second author at the University of Bielefeld in November 2006. He thanks the Department of
Mathematics and in particular the SFB 701 for its generous hospitality and Friedrich
G\"otze for the invitation and many interesting discussions.
A preliminary version of this paper appeared as preprint \cite{Speicher5}
of SFB 701.

The second author also thanks Uffe Haagerup for pointing out how ideas from \cite{Haagerup1} can be used to
improve the results from an earlier version of this paper.

\end{document}